\documentclass{amsart}
\usepackage{amsmath,amssymb,amsbsy,amscd,amsthm,mathrsfs,ulem}

\newtheorem{thm}{Theorem}[section]
\newtheorem{lem}[thm]{Lemma}
\newtheorem{cor}[thm]{Corollary}
\newtheorem{prop}[thm]{Proposition}

\theoremstyle{definition}
\newtheorem{defn}[thm]{Definition}
\newtheorem{example}[thm]{Example}

\theoremstyle{remark}
\newtheorem{rem}[thm]{Remark}

\numberwithin{equation}{section}

\newcommand{\nd }{\noindent}
\newcommand{\sm}{\setminus}

\newcommand{\zs}{\{ 0\} }
\newcommand{\bzs}{\{ \zero\} }
\newcommand{\zero}{{\bf{0}}}

\newcommand{\R}{{\bf R}}
\newcommand{\Q}{{\bf Q}}

\newcommand{\Z}{{\bf Z}}
\newcommand{\e}{{\bf e}}

\newcommand{\Zn}{{\Z _{\geq 0}}}
\newcommand{\Rn}{{\R _{\geq 0}}}
\newcommand{\Qn}{{\Q _{\geq 0}}}
\newcommand{\Zp}{{\Z _{>0}}}
\newcommand{\Rp}{{\R _{>0}}}

\newcommand{\x}{\boldsymbol{x}}
\newcommand{\y}{\boldsymbol{y}}

\newcommand{\w}{{\bf w}}

\newcommand{\kx}{k[\x ]}
\newcommand{\kxx}{k[\x ^{\pm 1}]}

\newcommand{\kxy}{k[\x ,\y]}

\newcommand{\cW }{\mathcal{W}}

\newcommand{\supp }{\mathop{\rm supp}\nolimits}
\newcommand{\rin }{\mathop{\rm in}\nolimits}
\newcommand{\rinp }{\mathop{\rm in}_{\preceq}\nolimits}
\newcommand{\rinpp }{\mathop{\rm in}_{\preceq '}\nolimits}
\newcommand{\degp }{\mathop{\rm deg}_{\preceq}\nolimits}
\newcommand{\degpp }{\mathop{\rm deg}_{\preceq '}\nolimits}

\newcommand{\id}{{\rm id}}

\newcommand{\p}{{\preceq }}

\def\mbi#1{\boldsymbol{#1}}

\newcommand{\ba}{{\mbi a}}
\newcommand{\bb}{{\mbi b}}
\newcommand{\bc}{{\mbi c}}
\newcommand{\bi}{{\mbi i}}
\newcommand{\bj}{{\mbi j}}

\begin{document}

\title[Polynomial subalgebras without finite SAGBI bases]{
A new class of finitely generated polynomial subalgebras without finite SAGBI bases}

\author{Shigeru Kuroda}
\address{Department of Mathematical Sciences\\
Tokyo Metropolitan University\\
1-1 Minami-Osawa, Hachioji, Tokyo, 192-0397, Japan}
\email{kuroda@tmu.ac.jp}
\thanks{This work is partly supported by JSPS KAKENHI
Grant Number 18K03219}


\date{}

\begin{abstract}
The notion of initial ideal for an ideal of a polynomial ring appears in the theory of Gr\"obner basis. Similarly to the initial ideals, we can define the initial algebra for a subalgebra of a polynomial ring, or more generally of a Laurent polynomial ring, which is used in the theory of SAGBI (Subalgebra Analogue to Gr\"obner Bases for Ideals) basis. The initial algebra of a finitely generated subalgebra is not always finitely generated, and no general criterion for finite generation is known. 

The aim of this paper is to present a new class of finitely generated subalgebras having non-finitely generated initial algebras. The class contains a subalgebra for which the set of initial algebras is continuum, as well as a subalgebra with finitely many distinct initial algebras. 
\end{abstract}

\maketitle

\section{Introduction}
\setcounter{equation}{0}

Let $k$ be a field, 
$\kxx =k[x_1,\ldots ,x_n,x_1^{-1},\ldots ,x_n^{-1}]$ 
the Laurent polynomial ring in $n$ variables over $k$, 
and $\Omega $ the set of total orders on $\Z ^n$ 
such that 
$\ba \preceq \bb $ implies 
$\ba +\bc \preceq \bb +\bc $ 
for each $\ba ,\bb ,\bc \in \Z ^n$. 
We fix $\p \in \Omega $. 
For $f=\sum _{\ba \in \Z ^n}u_{\ba }\x ^{\ba }\in \kxx $ 
with $u_{\ba }\in k$, 
we define 
$\supp f:=\{ \ba \in \Z ^n\mid u_{\ba }\ne 0\} $, 
and $\degp f:=\max _{\p }(\supp f)$ if $f\ne 0$, 
where $\x ^{\ba }:=x_1^{a_1}\cdots x_n^{a_n}$ 
for $\ba =(a_1,\ldots ,a_n)\in \Z ^n$. 
Then, we define 
$$
\degp V:=\{ \degp f\mid f\in V\sm \zs \} 
\quad\text{and}\quad 
\rinp V:=\sum _{\ba \in \degp V}k\x ^{\ba } 
$$
for each $V\subset \kxx $. 
If $A$ is a $k$-subalgebra of $\kxx $, 
then $\degp A$ is an additive submonoid of $\Z ^n$. 
Hence, 
$\rinp A$ is a $k$-subalgebra of $\kxx $, 
which we call the {\it initial algebra} of $A$. 
The notion of initial algebra and similar notions 
sometimes appear in the solution of deep problems in polynomial rings 
(cf.~\cite{Sugaku}).

Now, let $A$ be a $k$-subalgebra of $\kx :=k[x_1,\ldots ,x_n]$, 
and $\p $ a {\it monomial order} on $\kx $, 
i.e., an element of $\Omega $ 
with $\min _{\p }((\Zn )^n)=\zero $, 
where $\Zn $ denotes the set of nonnegative integers. 
Then, 
$S\subset A\sm \zs $ is called a {\it SAGBI basis} of $A$ 
if $\{ \rinp f\mid f\in S\} $ generates the $k$-algebra $\rinp A$. 
This notion was introduced by Robbiano-Sweedler~\cite{RS} 
as a Subalgebra Analogue to Gr\"obner Bases for Ideals (see also \cite{KM}). 
By definition, 
$A$ has a finite SAGBI basis if and only if 
the $k$-algebra $\rinp A$ is finitely generated, 
or equivalently, 
the monoid $\degp A$ is finitely generated. 
If this is the case, 
a theory similar to Gr\"obner basis theory works 
for SAGBI bases. 
However, 
even if $A$ is finitely generated, 
$\rinp A$ is not necessarily finitely generated. 
It is an important problem to find good criteria 
which guarantee finite generation for $\rinp A$, 
as proposed in \cite[Chap.\ 11]{Sturmfels}. 
For this purpose, 
it will be helpful to study the mechanism 
of non-finite generation of the initial algebras.

The best known result in this direction 
is the following: 
For a subgroup $G$ of the symmetric group $S_n$, 
consider the invariant ring 
\begin{equation}\label{eq:kx^G}
\kx ^G=\{ f\in \kx \mid 
f(x_{\sigma (1)},\ldots ,x_{\sigma (n)})=f(x_1,\ldots ,x_n)
\quad (\forall \sigma \in G)\} . 
\end{equation}
Then, 
$\rinp \kx ^G$ is finitely generated 
if and only if the group $G$ is generated by transpositions. 
This result is due to 
G\"obel~\cite{Gob} when $\preceq $ is the lexicographic order, 
and Kuroda~\cite{KurodaRIMS00}, \cite{Kuroda02} (see also \cite{JvdK}), 
Reichstein~\cite{Reich03} (see also \cite{Reich09}) and 
Thi\'{e}ry-Thomass\'{e}~\cite{TT} for general $\p $. 
Moreover, 
$\# \{ \rinp \kx ^G\mid \p \in \Omega _0\} $ 
is equal to $|G|$ 
if $G$ is generated by transpositions, 
and $\# \R $ otherwise, 
where $\Omega _0$ is the set of monomial orders on $\kx $ 
(cf.~ Kuroda's papers cited above, 
Tesemma~\cite{Tesemma} and Anderson\ et al.~\cite{Anderson}). 
This result is interesting because 
$\# \{ \rinp I\mid \p \in \Omega _0\} <\infty $ 
if $I$ is an ideal of $\kx $ 
(cf.~\cite{Schwartz}). 
We mention that 
\cite{KurodaRIMS00}, \cite{Kuroda02}, 
\cite{Reich03}, \cite{Reich09}, 
\cite{Tesemma} and \cite{Anderson} 
treated more general classes of invariant rings 
than (\ref{eq:kx^G}).

In this paper, 
we present a new construction of finitely generated 
$k$-subalgebras $\mathscr{A}$ of $\kxx $ 
for which $\{ \p \in \Omega \mid 
\rinp \mathscr{A}\text{ is not finitely generated}\} \ne \emptyset $. 
The construction is very general, 
and we obtain a large class of 
$k$-subalgebras with this property. 
The class contains a $k$-subalgebra $\mathscr{A}$ 
with $\# \{ \rinp \mathscr{A}\mid \p \in \Omega _0\} =\# \R $, 
as well as $\mathscr{A}$ with 
$\# \{ \rinp \mathscr{A}\mid \p \in \Omega \} <\infty $. 
Except for a class of invariant rings mentioned above, 
no such class of $k$-subalgebras are previously found. 
The technique we use to show the non-finite generation of 
initial algebras is different from that used for 
the invariant rings. 
In fact, 
as $k$-vector spaces, 
the invariant rings have $k$-bases 
with nice properties (cf.\ e.g.~\cite[Lem.\ 2.4]{Kuroda02}), 
but our $k$-subalgebras have no such structures in general.

This paper is organized as follows. 
In Section \ref{sect:main}, 
we describe the construction and state the main theorem. 
In Section \ref{sect:Prel}, 
we show some preliminary results, 
which are used in 
Section~\ref{sect:proof} to prove the main theorem. 
In Section~\ref{sect:example}, 
we give some examples, 
and shortly discuss the cardinality of the set of initial algebras. 

\medskip 

\nd {\bf Notation}\quad 
We define 
$\supp V:=\bigcup _{f\in V\sm \zs }\supp f$ for each $V\subset \kxx $, 
$K_{\ge 0}:=\{ a\in K\mid a\ge 0\} $ and $K_{>0}:=\{ a\in K\mid a>0\} $ 
for each $K\subset \R $, 
and 
$$
K_1\ba _1+\cdots +K_s\ba _s:=\{ \lambda _1\ba _1+\cdots +\lambda _s\ba _s\mid 
\lambda _1\in K_1,\ldots ,\lambda _s\in K_s\} 
$$
for each $K_i\subset \R $ and $\ba _i\in \R ^n$. 
The coordinate unit vectors of $\R ^n$ 
are denoted by $\e _1,\ldots ,\e _n$. 
For each $\p \in \Omega $ and $f\in \kxx \sm \zs $, 
we define $\rinp f:=u_{\ba }\x ^{\ba }$, 
where $\ba :=\degp f$ 
and $u_{\ba }\in k^*$ is the coefficient of 
$\x ^{\ba }$ in $f$. 
For $\p \in \Omega $ and $\ba ,\bb \in \Z ^n$, 
we write $\ba \prec \bb $ if $\ba \preceq \bb $ and $\ba \ne \bb $.

\section{Construction and main result}\label{sect:main}
\setcounter{equation}{0}

We begin with a lemma about commutative algebras. 
Let $k$ be a field, 
$B$ a commutative $k$-algebra, 
and $J$ an ideal of $B$. 
Assume that 
$$
B=k[\{ a_{i,j}\mid i=1,\ldots ,r,\ j=1,\ldots ,N\} \cup U], 
$$
where $r\ge 2$ and $N\ge 1$, 
$a_{i,j}$'s are elements of $B$ such that $a_{i_1,j}a_{i_2,j}\in J$ 
for each $1\le i_1<i_2\le r$ and $j=1,\ldots ,N$, 
and $U$ is a finite subset of $J$. 
We define 
$$
A:=k[\{ a_{1,j}+\cdots +a_{r,j}\mid 
j=1,\ldots ,N\} ]+J. 
$$

\begin{lem}\label{lem:basic}
\nd{\rm (i)} 
$B$ is integral over $A$. 

\nd{\rm (ii)} 
The $k$-subalgebra $A$ of $B$ is finitely generated. 
\end{lem}
\begin{proof}
For $j=1,\ldots ,N$, 
the monic polynomial 
$(x-a_{1,j})\cdots (x-a_{r,j})$ 
belongs to $A[x]$, 
since $a_{1,j}+\cdots +a_{r,j}\in A$, 
and $a_{i_1,j}a_{i_2,j}B\subset J\subset A$ for each $1\le i_1<i_2\le r$. 
Hence, 
$a_{1,j},\ldots ,a_{r,j}$ are integral over $A$. 
Clearly, 
elements of $U\subset J\subset A$ are integral over $A$, 
proving (i). 
Since the $k$-algebra $B$ is finitely generated, 
(ii) follows from (i) 
(cf.~\cite[Prop.\ 7.8]{AM}). 
\end{proof}

Now, 
let $C\subset \R ^n$ be a {\it convex cone}, i.e., 
a nonempty subset such that 
$\alpha \ba +\beta \bb \in C$ 
holds for all $\alpha ,\beta \ge 0$ and $\ba ,\bb \in C$. 
We call $F\subset C$ a {\it face} of $C$ 
if there exists $\omega \in \R ^n$ such that 
$F=\{ \ba \in C\mid \omega \cdot \ba =0\} $ and 
$\omega \cdot \ba \ge 0$ for all $\ba \in C$, 
where $\cdot $ denotes the standard inner product. 
We call this $\omega $ a {\it normal vector} of $F$. 
Take faces $C_1,\ldots ,C_r$ of $C$ 
such that $C_i\not\subset C_j$ if $i\ne j$, 
where $r\ge 2$. 
Then, we define 
\begin{equation}\label{eq:C^circ}
\begin{aligned}
C^{\circ }&:=C\sm \bigcup _{i=1}^rC_i
=\{ \ba \in C\mid \omega _i\cdot \ba >0\ (i=1,\ldots ,r)\} ,\\
C_i^{\circ }&:=C_i\sm \bigcup _{j\ne i}C_j
=\{ \ba \in C\mid \omega _i\cdot \ba =0,\ 
\omega _j\cdot \ba >0\ (\forall j\ne i)\} 
\end{aligned}
\end{equation}
for $i=1,\ldots ,r$, 
where $\omega _i$ is a normal vector of $C_i$. 
By definition, 
$C^{\circ }$, $C_1^{\circ },\ldots ,C_r^{\circ }$ 
are pairwise disjoint. 
We can check that 

\smallskip 

\nd (C1) $\ba \in C$ and $\bb \in C^{\circ }$ 
imply $\ba +\bb \in C^{\circ }$; 

\smallskip

\nd (C2) 
$\ba \in C_i^{\circ }$, 
$\bb \in C_j^{\circ }$ 
and $i\ne j$ 
imply $\ba +\bb \in C^{\circ }$.

\begin{example}\label{ex:basic}\rm 
$C:=(\Rn )^2$ is a convex cone, 
and $C_i:=\Rn \e _i$ for $i=1,2$ are faces of $C$. 
In this case, 
we have $C^{\circ }=(\Rp )^2$ 
and $C_i^{\circ }:=\Rp \e _i$ for $i=1,2$.

{\unitlength=0.5cm
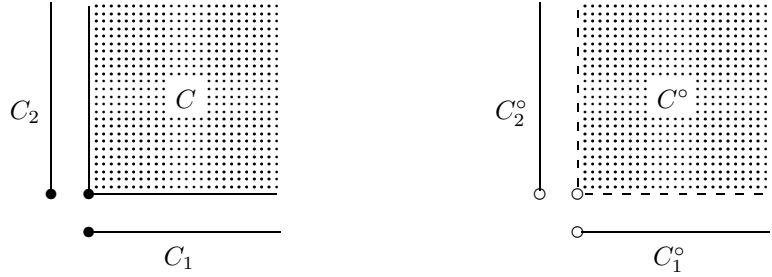
\begin{picture}(7,8)

\put(-0.2,2.2){
\multiput(4.2,0)(0,0.2){10}{\multiput(0,0.2)(0.2,0){24}{\circle*{0.1}}}
\multiput(4.2,3)(0,0.2){10}{\multiput(0,0.2)(0.2,0){24}{\circle*{0.1}}}
\multiput(4.2,0)(0,0.2){25}{\multiput(0,0.2)(0.2,0){10}{\circle*{0.1}}}
\multiput(7.2,0)(0,0.2){25}{\multiput(0,0.2)(0.2,0){10}{\circle*{0.1}}}

\put(4,0){\line(0,1){5}}
\put(4,0){\line(1,0){5}}
\put(4,0){\circle*{0.3}}

\put(3,0.1){\line(0,1){5}}
\put(3,0){\circle*{0.3}}
\put(1.9,2){$C_2$}

\put(4.1,-1){\line(1,0){5}}
\put(4,-1){\circle*{0.3}}
\put(6,-1.9){$C_1$}
\put(6.3,2.3){$C$}

\put(13,0){
\multiput(4.2,0)(0,0.2){10}{\multiput(0,0.2)(0.2,0){24}{\circle*{0.1}}}
\multiput(4.2,3)(0,0.2){10}{\multiput(0,0.2)(0.2,0){24}{\circle*{0.1}}}
\multiput(4.2,0)(0,0.2){25}{\multiput(0,0.2)(0.2,0){9}{\circle*{0.1}}}
\multiput(7.2,0)(0,0.2){25}{\multiput(0,0.2)(0.2,0){10}{\circle*{0.1}}}

\multiput(4,0.2)(0,0.5){10}{\line(0,1){0.2}}
\multiput(4.2,0)(0.5,0){10}{\line(1,0){0.2}}
\put(4,0){\circle{0.3}}

\put(3,0.1){\line(0,1){5}}
\put(3,0){\circle{0.3}}
\put(1.8,2){$C_2^{\circ }$}

\put(4.1,-1){\line(1,0){5}}
\put(4,-1){\circle{0.3}}
\put(6,-1.9){$C_1^{\circ }$}

\put(6.1,2.3){$C^{\circ }$}
}
}
\end{picture}}

\end{example}

We use the following objects to construct our $k$-algebra: 

\smallskip 

\nd (A1) 
A finitely generated $k$-domain $R$ and 
an ideal $I\ne \zs $ of $R$ 
such that $R=k+I$. 

\smallskip 

\nd (A2) 
Injective homomorphisms 
$\phi _1,\ldots ,\phi _r:R\to \kxx $ of $k$-algebras 
such that $\supp \phi _i(I)\subset C_i^{\circ }$ 
for $i=1,\ldots ,r$. 

\smallskip 

\nd (A3) 
A finite subset $U\subset \kxx $ such that $\supp U\subset C^{\circ }$ 
($U$ may be empty). 

\smallskip 

By (A1), 
there exist $N\ge 1$ and $r_1,\ldots ,r_N\in I\sm \zs $ 
such that 
$$
R=k[r_1,\ldots ,r_N]\quad\text{and}\quad 
I=\sum _{(i_1,\ldots ,i_N)\ne (0,\ldots ,0)}
kr_1^{i_1}\cdots r_N^{i_N}. 
$$
Moreover, 
since $\zero \not\in \supp \phi _1(I)$ by (A2), 
we see that $R=k\oplus I$. 
We define 
\begin{align*}
\mathscr{B}&:=k[\phi _1(I)\cup \cdots \cup \phi _r(I)\cup U]
=k[\{ \phi _i(r_j)\mid i=1,\ldots ,r,\ j=1,\ldots ,N\} \cup U], 
\end{align*}
and $J$ to be the ideal of $\mathscr{B}$ 
generated by 
$\bigcup _{1\le i_1<i_2\le r}\phi _{i_1}(I)\phi _{i_2}(I)\cup U$, 
or by 
\begin{equation}\label{eq:gen of J}
\{ 
\phi _{i_1}(r_{j_1})
\phi _{i_2}(r_{j_2})
\mid 
1\le i_1<i_2\le r,\ 
j_1,j_2\in \{ 1,\ldots ,N\} \} \cup U. 
\end{equation}
We define a $k$-linear map by 
$\phi :=\phi _1+\cdots +\phi _r:R\to \kxx $, 
and set 
\begin{equation}\label{eq:scrA}
\mathscr{A}:=k[\phi (r_1),\ldots ,\phi (r_N)]+J. 
\end{equation}
Then, 
we get the following theorem by Lemma~\ref{lem:basic} (ii).

\begin{thm}\label{thm:A f.g}
The $k$-algebra $\mathscr{A}$ is finitely generated. 
\end{thm}

Note that, 
for each $p,q\in I$, 
we have 
\begin{align*}
\phi (p)\phi(q)
&=(\phi _1(p)+\cdots +\phi _r(p))
(\phi _1(q)+\cdots +\phi _r(q)) \\
&\equiv \phi _1(p)\phi _1(q)+\cdots +\phi _r(p)\phi _r(q)
=\phi (pq)\pmod{J}. 
\end{align*}
Hence, we can write (\ref{eq:scrA}) as 
\begin{equation}\label{eq:decomp}
\mathscr{A}
=k+\sum _{(i_1,\ldots ,i_N)\ne (0,\ldots ,0)}
k\phi (r_1^{i_1}\cdots r_N^{i_N})+J
=k+\phi (I)+J. 
\end{equation}
Therefore, 
$\mathscr{A}$ is also defined 
independently of the choice of $r_1,\ldots ,r_N$.

The following theorem is the main result of this paper. 

\begin{thm}\label{thm:main}
Let $\p \in \Omega $ be such that 
$\min _{\p }(\degp \phi _i(R))=\zero $ for some $1\le i\le r$. 
Then, $\degp \mathscr{A}$ is not finitely generated. 
\end{thm}

By the definition of monomial order, 
Theorem~\ref{thm:main} implies the following corollary.

\begin{cor}\label{cor:main}
If there exists $1\le i\le r$ such that 
$\phi _i(R)\subset k[x_1,\ldots ,x_n]$, 
then $\degp \mathscr{A}$ is not finitely generated 
for any monomial order $\p \in \Omega $. 
\end{cor}

In the rest of this section, 
we discuss basic structure of $\degp \mathscr{A}$.

\begin{lem}\label{lem:J}
We have $\supp \mathscr{B}\subset C$ 
and $\supp J\subset C^{\circ }$. 
\end{lem}
\begin{proof}
If $f\in \phi _1(I)\cup \cdots \cup \phi _r(I)\cup U$, 
then we have $\supp f\subset C$ by (A2) and (A3). 
This implies $\supp \mathscr{B}\subset C$. 
If $g$ is in (\ref{eq:gen of J}), 
then we have $\supp g\subset C^{\circ }$ by (C2) and (A3). 
Since $\supp \mathscr{B}\subset C$, 
it follows from (C1) that 
$\supp gh\subset C^{\circ }$ 
for any $h\in \mathscr{B}$. 
This implies $\supp J\subset C^{\circ }$. 
\end{proof}

Now, by (\ref{eq:decomp}), 
each $f\in \mathscr{A}$ is written as 
\begin{equation}\label{eq:f in A}
f=c+\phi (p)+g,\ 
\text{ where }\ c\in k,\ p\in I\ \text{ and }\ g\in J.
\end{equation}
Then, 
$\supp c$, $\supp \phi (p)$ and $\supp g$ 
are pairwise disjoint, 
since 
$\supp c\subset \bzs $, $\supp \phi (p)\subset C_1^{\circ }
\cup \cdots \cup C_r^{\circ }$ 
and $\supp g\subset C^{\circ }$ by Lemma~\ref{lem:J}, 
and $\bzs $, 
$C_1^{\circ }\cup \cdots \cup C_r^{\circ }$ 
and $C^{\circ }$ are pairwise disjoint. 
Therefore, for each $\p \in \Omega $, we have 
\begin{equation}\label{eq:deg A}
\degp \mathscr{A}=\bzs \sqcup \degp \phi (I)\sqcup \degp J
\ \text{ and }\ 
\rinp \mathscr{A}=k\oplus \rinp \phi (I)\oplus \rinp J. 
\end{equation}

\begin{example}\label{ex:RS}\rm
Let $C$, $C_1$ and $C_2$ be as in Example~\ref{ex:basic}. 
Let 
$R:=k[x]$ be the polynomial ring in one variable, 
$I:=xk[x]$, 
$\phi _i:k[x]\ni p(x)\mapsto p(x_i)\in \kxx $ for $i=1,2$, 
and $U:=\emptyset $. 
Then, 
we have $\mathscr{B}=k[x_1,x_2]$, 
$J=(x_1x_2)$, and 
$$
\mathscr{A}=k[x_1+x_2]+J
=k+\sum _{i\ge 1}k(x_1^i+x_2^i)+J
$$
by (\ref{eq:decomp}). 
For $\p \in \Omega $ with $\e _1\succ \e _2$, we have 
$\degp \mathscr{A}=\bzs \sqcup \Zp \e _1\sqcup (\Zp )^2$. 
This monoid is not finitely generated.

{\unitlength=0.5cm
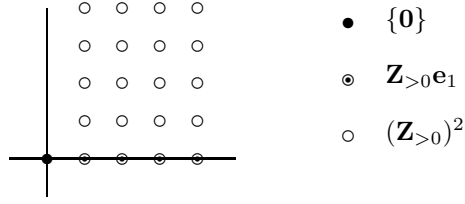
\begin{picture}(7,6)

\put(3,1){

\put(4,-1){\line(0,1){5}}
\put(3,0){\line(1,0){6}}
\put(4,0){\circle*{0.3}}

\multiput(5,1)(0,1){4}{\multiput(0,0)(1,0){4}{\circle{0.3}}}

\multiput(5,0)(1,0){4}{\circle*{0.15}}
\multiput(5,0)(1,0){4}{\circle{0.3}}

\put(13,3.5){$\bzs $}

\put(13,2){$\Zp \e _1$}

\put(13,0.5){$(\Zp )^2$}

\put(12,3.6){\circle*{0.3}}

\put(12,2.1){\circle{0.3}}
\put(12,2.1){\circle*{0.15}}

\put(12,0.6){\circle{0.3}}

}

\end{picture}}
\end{example}

In \cite{RS}, 
Robbiano-Sweedler 
gave $A=k[x_1+x_2,x_1x_2,x_1x_2^2]$ 
as an example of a finitely generated $k$-subalgebra 
whose initial algebra is not finitely generated. 
We note that $\mathscr{A}$ in Example~\ref{ex:RS} 
is equal to this $A$.

\begin{rem}\label{prop:non-normal}\rm 
{\rm (i)} 
$\mathscr{A}$ is not equal to $\mathscr{B}$ 
for the following reason: 
For each $p\in I\sm \zs $, 
we have 
$\supp \phi (p)=\bigsqcup _{i=1}^r\supp \phi _i(p)\not\subset C_1^{\circ }$, 
since $C_1^{\circ },\ldots ,C_r^{\circ }$ are pairwise disjoint. 
Hence, 
$\supp f\not\subset C_1^{\circ }$ holds for any $f\in \mathscr{A}\sm \zs $ 
by (\ref{eq:f in A}). 
On the other hand, 
we have $\phi _1(r_1)\in \mathscr{B}\sm \zs $ and 
$\supp \phi _1(r_1)\subset C_1^{\circ }$.

\nd 
{\rm (ii)} 
We have $Q(\mathscr{A})=k(J)=Q(\mathscr{B})$, 
since $g\mathscr{A}\subset g\mathscr{B}\subset J$ 
for $0\ne g\in J$. 
Here, $Q(\cdot )$ denotes the field of fractions. 
Therefore, 
$\mathscr{A}$ is not normal by (i) and Lemma~\ref{lem:basic}~(i). 
\end{rem}

\begin{rem}\label{rem:rat poly cone}\rm 
In the construction above, 
we may assume without loss of generality that 
$C$ is a rational polyhedral cone (see Section~\ref{sect:Prel}) 
by replacing $C$ with 
$C':=\sum _{\ba \in \mathscr{F}}\Rn \ba $, 
where $\mathscr{F}:=\bigcup _{i=1}^r\bigcup _{j=1}^N
\supp \phi _i(r_j)\cup \supp U$. 
Actually, 
since $C'\subset C$, we see that 
$C_i':=\{ \ba \in C'\mid \omega _i\cdot \ba =0\} $ 
is a face of $C'$ for $i=1,\ldots ,r$, 
and we have 
$\supp \phi _i(I)\subset C_i'\sm \bigcup _{j\ne i}C_j'$ 
for each $i$ 
and $\supp U\subset C'\sm \bigcup _{j=1}^rC_j'$. 
\end{rem}

\section{Preliminary}\label{sect:Prel}
\setcounter{equation}{0}

Recall that a convex cone $C\subset \R ^n$ 
is {\it polyhedral} if 
$C=\bigcap _{i=1}^t
\{ \ba \in \R ^n\mid \eta _i \cdot \ba \ge 0\} $ 
for some $t\ge 0$ and 
$\eta _1,\ldots ,\eta _t\in \R ^n\sm \bzs $. 
If this is the case, 
$C$ is a closed subset of $\R ^n$ for the Euclidean topology. 
By the Farkas-Minkowski-Weyl theorem (cf.~\cite[Cor.~7.1a]{LP}), 
a convex cone $C\subset \R ^n$ is polyhedral 
if and only if {\it finitely generated}, 
i.e., $C=\Rn \ba _1+\cdots +\Rn \ba _s$ 
for some $s\ge 0$ and $\ba _1,\ldots ,\ba _s\in C$.

For each $S\subset \R ^n$, 
we denote by $\Rn S$ the convex cone generated by $S$, 
i.e., 
the union of 
$\Rn \ba _1+\cdots +\Rn \ba _s$ 
for all $s\ge 0$ and $\ba _1,\ldots ,\ba _s\in S$. 
If $\Rn S$ is finitely generated, 
then $\Rn S=\Rn S_0$ holds for some finite subset $S_0$ of $S$. 
Actually, 
for any $s\ge 0$ and $\ba _1,\ldots ,\ba _s\in \Rn S$, 
there exists a finite subset $S_0$ of $S$ 
such that $\ba _1,\ldots ,\ba _s\in \Rn S_0$.

We call $C\subset \R ^n$ a {\it rational polyhedral cone} 
if $C=\Rn \ba _1+\cdots +\Rn \ba _s$ 
for some $s\ge 0$ and $\ba _1,\ldots ,\ba _s\in \Q ^n$. 
When this is the case, 
$\Q ^n\cap C$ is equal to $\Qn \ba _1+\cdots +\Qn \ba _s$. 
In fact, 
to see this, 
we may assume that $\ba _1,\ldots ,\ba _s$ are linearly independent 
by Carath\'eodory's theorem (cf.\ e.g., \cite[Cor.~7.1i]{LP}), 
and hence are part of a $\Q $-basis of $\Q ^n$.

The following lemma is a variant of Gordan's lemma.

\begin{lem}\label{lem:Gordan}
Let $S$ be a submonoid of $\Z ^n$ 
such that $\Rn S=\sum _{i=1}^s\Rn \ba _i$ 
for some $s\ge 0$ and 
$\ba _1,\ldots ,\ba _s\in S$. 
Then, 
$S=\sum _{i=1}^s\Zn \ba _i+F$ holds for some 
finite subset $F$ of $S$. 
In particular, 
$S$ is finitely generated. 
\end{lem}
\begin{proof}
Note that $A:=\sum _{\ba \in S}k\x ^{\ba }$ 
is a $k$-subalgebra of $\kxx $ 
containing $R:=k[\x ^{\ba _1},\ldots ,\x ^{\ba _s}]$. 
We show that $A$ is a finite $R$-module, 
i.e., 
$A=\sum _{i=1}^uRf_i$ 
for some $u\ge 1$ and $f_1,\ldots ,f_u\in A$. 
Then, 
$S=\sum _{i=1}^s\Zn \ba _i+F$ holds with $F:=\bigcup _{i=1}^u\supp f_i$.

By assumption, 
$\Rn S$ is a rational polyhedral cone. 
Hence, 
the monoid $\widehat{S}:=\Z ^n\cap \Rn S$ 
is finitely generated by Gordan's lemma 
(cf.~\cite[Prop.\ 1.1 (ii)]{Oda}). 
Choose $\bb _1,\ldots ,\bb _t\in \widehat{S}$ 
so that $\widehat{S}=\sum _{i=1}^t\Zn \bb _i$. 
Since $\bb _1,\ldots ,\bb _t\in \widehat{S}\subset \Q ^n\cap \Rn S
=\Q^n\cap \sum _{i=1}^s\Rn \ba _i=\sum _{i=1}^s\Qn \ba _i$, 
there exists $l\in \Zp $ 
such that $l\bb _1,\ldots ,l\bb _t\in \sum _{i=1}^s\Zn \ba _i$. 
This implies 
$(\x ^{\bb _1})^l,\ldots ,(\x ^{\bb _t})^l\in R$. 
Hence, $\x ^{\bb _1},\ldots ,\x ^{\bb _t}$ are integral over $R$. 
Thus, 
$k[\x ^{\bb _1},\ldots ,\x ^{\bb _t}]$ 
is a Noetherian $R$-module. 
Since $S\subset \widehat{S}$, 
we have $A\subset k[\x ^{\bb _1},\ldots ,\x ^{\bb _t}]$. 
Therefore, 
$A$ is a finite $R$-module. 
\end{proof}

Let $\cW $ be the set of 
$(w_1,\ldots ,w_n)\in \R ^n$ such that 
$w _1,\ldots ,w_n$ are linearly independent over $\Q $. 
This set is the complement of 
a countable union of nowhere dense subsets 
$\{ \w \in \R ^n\mid \w \cdot \ba =0\} $ of $\R ^n$, 
where $\ba \in \Q ^n\sm \bzs $. 
Hence, 
$\cW $ is dense in $\R ^n$ by Baire's theorem.

Lemma~\ref{lem:approx} 
below is well known when $\p $ is a monomial order. 
We include a proof here, 
because we could not find a suitable reference 
dealing with general $\p \in \Omega $.

\begin{lem}\label{lem:approx}
For any $\p \in \Omega $, $s\ge 1$ and 
$\ba _1,\ldots ,\ba _s\in \Z ^n$ 
with $\ba _1,\ldots ,\ba _s\succ \zero $, 
there exists $\w \in \cW $ such that 
$\w \cdot \ba _1>0,\ldots ,\w \cdot \ba _s>0$. 
\end{lem}
\begin{proof}
Since $\cW $ is dense and 
$V:=\{ \w \in \R ^n\mid 
\w \cdot \ba _1>0,\ldots ,\w \cdot \ba _s>0\} $ 
is open in $\R ^n$, 
it suffices to check $V\ne \emptyset $, 
or equivalently $\bzs $ is a face of 
$C:=\sum _{i=1}^s\Rn \ba _i$. 
If not, 
then $C$ is not {\it strongly convex} (cf.\ \cite[Chap.\ 1.1]{Oda}), 
so there exists $\alpha =(\alpha _i)_{i=1}^s\in 
(\Rn )^s\sm \bzs $ 
such that $\sum _{i=1}^s\alpha _i\ba _i=\zero $. 
Since $\ba _1,\ldots ,\ba _s\in \Z ^n$, 
we may choose $\alpha $ from $(\Zn )^s\sm \bzs $. 
Then, 
we have 
$\zero =\sum _{i=1}^s\alpha _i\ba _i\succ \zero $, 
a contradiction. 
\end{proof}

The following lemma is obtained 
by using Lemma~\ref{lem:approx} 
for $\{ \ba _1,\ldots ,\ba _s\} :=
\{ \bb -\ba \mid \ba ,\bb \in \mathscr{F},\ \ba \prec \bb \} $.

\begin{lem}\label{lem:approx lem}
For every finite subset $\mathscr{F}$ of $\Z ^n$ 
and $\p \in \Omega $, 
we have 
$$
\cW _{\mathscr{F},\p }:=
\{ \w \in \cW \mid 
\w \cdot \ba < \w \cdot \bb \iff 
\ba \prec \bb \quad (\forall \ba ,\bb \in \mathscr{F})\} 
\ne \emptyset . 
$$
\end{lem}

We remark that the same argument as above shows, 
for every finite subset $\mathscr{F}$ of $\Z ^n$ 
and $\p \in \Omega $, 
there exists $\w \in \Q ^n$ such that 
$\w \cdot \ba < \w \cdot \bb \Leftrightarrow  
\ba \prec \bb $ 
$(\forall \ba ,\bb \in \mathscr{F})$, 
since $\Q ^n$ is also dense in $\R ^n$.

The following are also used to prove 
Theorem~\ref{thm:main}.

\begin{lem}\label{lem:well-ordered}
Let $V$ be a $k$-vector subspace of $\kxx $, 
$\p '\in \Omega $, and $S\subset \degpp V$. 
If $S$ is well-ordered for $\p '$, 
and $S\cap \supp f\ne \emptyset $ for each $f\in V\sm \zs $, 
then we have $S=\degpp V$. 
\end{lem}
\begin{proof}
Suppose that there exists $\ba \in \degpp V$ 
not belonging to $S$. 
Pick $f\in V\sm \zs $ with $\degpp f=\ba $. 
Then, 
$\ba _f:=\max _{\p '}(S\cap \supp f)$ is less than $\ba $. 
Since $S$ is well-ordered, 
we can choose $f$ so that $\ba _f$ is minimal. 
Since $\ba _f\in S\subset \degpp V$, 
there exists $g\in V\sm \zs $ with $\rinpp g=\x ^{\ba _f}$. 
Set $h:=f-cg$, 
where $c\in k^*$ is the coefficient of $\x ^{\ba _f}$ in $f$. 
Then, 
we have 
$\ba _h\prec '\ba _f$, 
$h\in V$, 
and $\degpp h=\ba $, 
since $\degpp g=\ba _f\prec '\ba =\degpp f$. 
This contradicts the minimality of $\ba _f$. 
\end{proof}

For each $\w \in \cW $, 
we define $\p \in \Omega $ by $\ba \preceq \bb $ 
if $\w \cdot \ba \le \w \cdot \bb $ 
for $\ba ,\bb \in \Z ^n$. 
We denote this $\p $ by $o(\w )$, 
and set $o(\mathcal{V}):=\{ o(\w )\mid \w \in \mathcal{V}\} $ 
for $\mathcal{V}\subset \cW $.

\begin{prop}\label{prop:well-ordered}
Let $P$ be a $k$-subalgebra of $\kxx $, 
and let $\p \in \Omega $ be such that $\degp P$ is finitely generated 
and $\min _{\p }(\degp P)=\zero $. 
Then, 
there exists a finite subset $\mathscr{F}$ of $\Z ^n$ 
for which the following $(*)$ holds for each 
$\p '\in o(\cW _{\mathscr{F},\p })$$:$

\nd $(*)$ $\degp P$ is well-ordered for $\p '$ 
and $\deg _{\p '}P=\degp P$. 

\end{prop}
\begin{proof}
By assumption, 
there exist $s\ge 0$ and 
$f_1,\ldots ,f_s\in P\sm \zs $ such that 
$\degp P=\sum _{i=1}^s\Zn \degp f_i$ and 
$\degp f_i\succ \zero $ for each $i$. 
Set $\mathscr{F}:=\bzs \cup \bigcup _{i=1}^s\supp f_i$. 
Take any $\w \in \cW _{\mathscr{F},\p }$ and put $\p ':=o(\w )$. 
Then, 
$\w \cdot \degp f_i>\w \cdot \zero =0$ holds for each $i$. 
Hence, 
for any $\bb \in \degp P$, 
the number of $\ba \in \degp P=\sum _{i=1}^s\Zn \degp f_i$ 
with $\w \cdot \ba <\w \cdot \bb $ is finite. 
This implies that $\degp P$ 
is well-ordered for $\p '$.

Since $\degp f_i=\deg _{\p '}f_i\in \deg _{\p '}P$ 
by the definition of $\mathscr{F}$, 
we have 
$\degp P=\sum _{i=1}^s\Zn \degp f_i\subset \deg _{\p '}P$. 
We also have 
$\supp f\cap \degp P\ne \emptyset $ 
for each $f\in P\sm \zs $. 
Therefore, 
we get $\degp P=\deg _{\p '}P$ by Lemma~\ref{lem:well-ordered} with 
$V:=P$ and $S:=\degp P$. 
\end{proof}

\section{Proof of Theorem~\ref{thm:main}}\label{sect:proof}
\setcounter{equation}{0}

The goal of this section is to prove Theorem~\ref{thm:main}. 
Throughout, 
let $R$ and $I$ be as in (A1), 
and let $\psi _1,\ldots ,\psi _r:R\to \kxx $ 
be injective homomorphisms of $k$-algebras, 
where $r\ge 2$. 
We assume that $\supp \psi _i(R)=\bzs \sqcup \supp \psi _i(I)$ 
for each $i$, 
and $\supp \psi _i(I)\cap \supp \psi _j(I)=\emptyset $ 
if $i\ne j$. 
These assumptions are fulfilled for 
$\phi _1,\ldots ,\phi _r$ in (A2). 
For the moment, we fix $\p \in \Omega $. 
We remark that 
$\degp \psi _i(R)=\bzs \sqcup \degp \psi _i(I)$ for each $i$.

Set $\psi :=\psi _1+\cdots +\psi _r:R\to \kxx $. 
For $i=1,\ldots ,r$, 
we define 
$$
I_{\p ,i}:=\{ p\in I\sm \zs \mid \degp \psi _i(p)\succ \degp \psi _j(p)
\ (\forall j\ne i)\} . 
$$
Then, 
since $\supp \psi _i(I)\cap \supp \psi _j(I)=\emptyset $ if $i\ne j$, 
we have $I\sm \zs =\bigsqcup _{i=1}^rI_{\p ,i}$, 
and 
\begin{equation}\label{eq:Ipi}
p\in I_{\p ,i}\iff \degp \psi (p)=\degp \psi _i(p)
\end{equation}
for $p \in I\sm \zs $ and $i\in \{ 1,\ldots ,r\}$. 
We remark that the following hold for each $i$. 

\smallskip 

\nd (1$^\circ $) 
$\degp \psi _i(I)\cap \degp \psi (I)=\degp \psi _i(I_{\p ,i})$. 

\smallskip 

\nd (2$^\circ $) 
$I_{\p ,i}$ is closed under multiplication, 
so $\degp \psi _i(I_{\p ,i})$ is closed under addition. 
Since 
$\degp \psi _i(I_{\p ,i})\subset \supp \psi _i(I)
\subset \Z ^n\sm \bzs $, 
it follows that $\# \degp \psi _i(I_{\p ,i})=\infty $ 
if $I_{\p ,i}\ne \emptyset $.

\begin{defn}\rm 
(i) We say that $(\psi _i)_{i=1}^r$ is $\p $-{\it complete} if, 
for $i=1,\ldots ,r$, 
the cone $\Rn (\degp \psi _i(R))$ is generated by 
a finite subset of 
$\degp \psi _i(I_{\p ,i})$, 
or equivalently 
$\Rn (\degp \psi _i(R))$ is equal to 
$\Rn (\degp \psi _i(I_{\p ,i}))$ 
and is finitely generated. 

\nd (ii) We say that $(\psi _i)_{i=1}^r$ is $\p $-{\it incomplete} 
if $(\psi _i)_{i=1}^r$ is not $\p $-complete. 
\end{defn}

If $(\psi _i)_{i=1}^r$ is $\p $-complete, 
then the following are true. 

\smallskip 

\nd (3$^\circ $) 
By Lemma~\ref{lem:Gordan}, 
$\degp \psi _i(R)$ is finitely generated 
for $i=1,\ldots ,r$.

\smallskip

\nd (4$^\circ $) 
Put $\psi ':=\psi _1+\psi _2$. 
Then, $\degp \psi _i(I)\cap \degp \psi (I)\subset 
\degp \psi _i(I)\cap \degp \psi '(I)$ 
holds for $i=1,2$. 
Hence, 
$(\psi _i)_{i=1}^2$ is also $\p $-complete 
in view of (1$^{\circ }$).

\smallskip

Now, let $\mathscr{A}$ be the $k$-subalgebra of $\kxx $ 
constructed in Section~\ref{sect:main}, 
and $(\phi _i)_{i=1}^r$ as in (A2).

\begin{thm}\label{thm:incomplete}
If $\degp \mathscr{A}$ is finitely generated, 
then $(\phi _i)_{i=1}^r$ is $\p $-complete. 
\end{thm}
\begin{proof}
Fix $i\in \{ 1,\ldots ,r\} $. 
Since $\Rn (\degp \mathscr{A})\subset \Rn (\supp \mathscr{B})\subset C$ 
by Lemma~\ref{lem:J}, 
we see that 
$C_i\cap \Rn (\degp \mathscr{A})$ 
is a face of $\Rn (\degp \mathscr{A})$ 
with normal vector $\omega _i$. 
Since $\degp \mathscr{A}$ 
is finitely generated by assumption, 
the cone $\Rn (\degp \mathscr{A})$ is finitely generated, 
and hence polyhedral. 
Thus, 
the cone 
$C_i\cap \Rn (\degp \mathscr{A})$ is also polyhedral, 
and therefore finitely generated. 
We show that
\begin{equation}\label{eq:finite generation}
\degp \phi _i(I)\subset 
C_i\cap \Rn (\degp \mathscr{A})
=\Rn (\degp \phi _i(I_{\p ,i})). 
\end{equation}
Then, 
it follows that $\Rn (\degp \phi _i(R))$ 
is equal to $\Rn (\degp \phi _i(I_{\p ,i}))$ 
and is finitely generated, 
since $\degp \phi _i(R)=\bzs \sqcup \degp \phi _i(I)
\supset \degp \phi _i(I_{\p ,i})$,

We have 
$\degp \phi _i(I)\subset 
\supp \phi _i(I)\subset C_i$ by (A2). 
We check 
$\degp \phi _i(I)\subset \Rn (\degp \mathscr{A})$. 
Pick any $\ba \in \degp \phi _i(I)$, 
and $f\in \phi _i(I)\sm \zs $ with $\degp f=\ba $. 
Since $r\ge 2$, 
we can find $j\in \{ 1,\ldots ,r\} \sm \{ i\} $ 
and $g\in \phi _j(I)\sm \zs $. 
Then, for each $l\ge 1$, 
we have $f^lg\in J\subset \mathscr{A}$. 
Hence, 
$\ba +l^{-1}\degp g=l^{-1}\degp f^lg$ belongs to 
$\Rn (\degp \mathscr{A})$. 
Since $\Rn (\degp \mathscr{A})$ is a polyhedral cone, 
$\Rn (\degp \mathscr{A})$ is a closed subset of $\R ^n$. 
Thus, 
$\lim _{l\to \infty }(\ba +l^{-1}\degp g)=\ba $ 
belongs to $\Rn (\degp \mathscr{A})$. 
This proves ``$\subset $" in (\ref{eq:finite generation}).

Next, 
note that $C_i\cap \degp \phi (I)\subset 
C_i\cap \bigcup _{j=1}^r\degp \phi _j(I)=\degp \phi _i(I)
\subset C_i$, 
since 
$C_i\cap \degp \phi _j(I)
\subset C_i\cap C_j^{\circ }
=\emptyset $ for any $j\ne i$. 
Hence, we get 
$C_i\cap \degp \phi (I)
=\degp \phi _i(I)\cap \degp \phi (I)
=\degp \phi _i(I_{\p ,i})$ by (\ref{eq:Ipi}). 
Since $C_i\cap \degp J\subset C_i\cap C^{\circ }=\emptyset $ 
by Lemma~\ref{lem:J}, 
we see from (\ref{eq:deg A}) that 
$C_i\cap \degp \mathscr{A}
=\bzs \cup \degp \phi _i(I_{\p ,i})$. 
This implies 
``$=$" in (\ref{eq:finite generation}). 
\end{proof}

Thanks to Theorem~\ref{thm:incomplete}, 
we are reduced to proving the following proposition.

\begin{prop}\label{prop:key}
If $\min _{\p }(\degp \psi _1(R))=\zero $, 
then $(\psi _i)_{i=1}^r$ is $\p $-incomplete. 

\end{prop}

We prove this proposition in the remainder of this section. 
First, we show two fundamental lemmas.

\begin{lem}[Finiteness principle]\label{lem:FP}
If $\p \in o(\cW )$, 
then the following assertions hold. 

\nd {\rm (i)} 
$(\degp \psi _1(R)\sm \degp \psi (I))\cap \sum _{i=1}^s\Rn \ba _i$ 
is a finite set for any $s\ge 1$ and 
$\ba _1,\ldots ,\ba _s\in \Rn (\degp \psi _1(I_{\p ,1}))$.

\nd {\rm (ii)} 
If the cone $\Rn (\degp \psi _1(R))$ is 
generated by a finite subset of 
$\degp \psi _1(I_{\p ,1})$, 
then $\degp \psi _1(R)\sm \degp \psi (I)$ is a finite set. 
\end{lem}
\begin{proof}
(i) Since $\ba _1,\ldots ,\ba _s\in \Rn (\degp \psi _1(I_{\p ,1}))$, 
there exists 
a finite subset $\mathscr{F}$ of $\degp \psi _1(I_{\p ,1})$ 
such that 
$\ba _1,\ldots ,\ba _s\in \Rn \mathscr{F}$. 
Since $\sum _{i=1}^s\Rn \ba _i\subset \Rn \mathscr{F}$, 
it suffices to show 
that 
$(\degp \psi _1(R)\sm \degp \psi (I))\cap \Rn \mathscr{F}$ 
is a finite set. 
Hence, 
by replacing $\{ \ba _1,\ldots ,\ba _s\} $ with $\mathscr{F}$, 
we may assume 
$\ba _1,\ldots ,\ba _s\in \degp \psi _1(I_{\p ,1})$. 
Set $S:=\degp \psi _1(R)\cap \sum _{i=1}^s\Rn \ba _i$. 
Then, 
since $\ba _1,\ldots ,\ba _s\in S$, 
we have $\Rn S=\sum _{i=1}^s\Rn \ba _i$. 
Hence, by Lemma~\ref{lem:Gordan}, 
there exists a finite subset $F\subset S\sm \bzs \subset 
\degp \psi _1(R)\sm \bzs =\degp \psi _1(I)$ 
such that $S=\sum _{i=1}^s\Zn \ba _i+F$. 
Now, let $\w \in \cW $ be such that $\p =o(\w )$. 
For $i=1,\ldots ,s$, 
pick $p_i\in I_{\p ,1}$ with $\degp \psi _1(p_i)=\ba _i$, 
and set $\ba _{i,j}:=\degp \psi _j(p_i)$ for $j=2,\ldots ,r$. 
Then, 
$\w \cdot (\ba _i-\ba _{i,j})>0$ holds for each $i$ and $j$, 
since $\ba _{i,j}\prec \ba _i$. 
For each $\bb \in F$, 
pick $q_{\bb }\in I\sm \zs $ 
with $\degp \psi _1(q_{\bb })=\bb $, 
and set $\bb _j:=\degp \psi _j(q_{\bb })$ for $j=2,\ldots ,r$. 
Then, 
there exists $u\in \Zp $ such that, 
for any $i=1,\ldots ,s$, $j=2,\ldots ,r$ and $\bb \in F$, 
we have $\w\cdot (\ba _i-\ba _{i,j})>u^{-1}|\w \cdot (\bb -\bb _j)|$.

Now, take any 
$l_1,\ldots ,l_s\in \Zn $ with $\sum _{i=1}^sl_i\ge u$ 
and $\bb \in F$. 
We show that $\bc :=\sum _{i=1}^sl_i\ba _i+\bb $ 
belongs to $\degp \psi (I)$. 
Then, 
it follows that 
$S\sm \degp \psi (I)
=(\sum _{i=1}^s\Zn \ba _i+F)\sm \degp \psi (I)$ 
is a finite set. 
For $p:=p_1^{l_1}\cdots p_s^{l_s}q_{\bb }\in I\sm \zs $, 
we have $\degp \psi _1(p)=\bc $, 
and $\bc _j:=\degp \psi _j(p)=\sum _{i=1}^sl_i\ba _{i,j}+\bb _j$ 
for $j=2,\ldots ,r$. 
Since 
\begin{align*}
\w \cdot (\bc -\bc _j)
&=\sum _{i=1}^sl_i\w \cdot (\ba _i-\ba _{i,j})+\w \cdot (\bb -\bb _j)\\
&>\sum _{i=1}^sl_iu^{-1}|\w \cdot (\bb -\bb _j)|+\w \cdot (\bb -\bb _j)
\ge 0
\end{align*}
by the choice of $u$, 
we get 
$\bc =\max _{\p }\{ \bc ,\bc _2,\ldots ,\bc _r\} 
=\degp \psi (p)\in \degp \psi (I)$.

(ii) The assertion follows from (i). 
\end{proof}

Next, 
for each $\ba \in \degp \psi _2(I_{\p ,2})$, 
we define 
$$
M^{\ba }_{\p }:=\{ \degp \psi _1(p)\mid p\in I_{\p ,2},\ 
\degp \psi _2(p)=\ba \} \subset \degp \psi _1(I). 
$$ 
By definition, 
$\bb \prec \ba $ holds for each $\bb \in M^{\ba }_{\p }$.

\begin{lem}[Injectivity principle]\label{lem:IP}
If there exists $\mu (\ba ):=\min _{\p }M^{\ba }_{\p }$ for each 
$\ba \in \degp \psi _2(I_{\p ,2})$, 
then the following assertions hold.

\nd {\rm (i)} 
$\mu (\ba )\not\in \degp \psi (I)$ 
for each $\ba \in \degp \psi _2(I_{\p ,2})$.

\nd {\rm (ii)} 
The map 
$\mu :\degp \psi _2(I_{\p ,2})\to \degp \psi _1(R)\sm \degp \psi (I)$ 
is injective. 

\nd {\rm (iii)} 
If $I_{\p ,2}\ne \emptyset $, 
then $\degp \psi _1(R)\sm \degp \psi (I)$ is an infinite set. 
\end{lem}
\begin{proof}
(i) Suppose that 
$\mu (\ba )\in \degp \psi (I)$ 
for some $\ba \in \degp \psi _2(I_{\p ,2})$. 
Then, 
$\mu (\ba )$ belongs to 
$\degp \psi _1(I)\cap \degp \psi (I)=\degp \psi _1(I_{\p ,1})$ 
by (1$^\circ $). 
Choose $p\in I_{\p ,2}$ and $q\in I_{\p ,1}$ 
so that $\degp \psi _2(p)=\ba $ and 
$\rinp \psi _1(p)=\rinp \psi _1(q)=\x ^{\mu (\ba )}$. 
Then, 
$\degp \psi _1(p-q)=\degp (\psi _1(p)-\psi _1(q))$ 
is less than $\mu (\ba )$. 
We claim that 
$p-q\in I_{\p ,2}$ and $\degp \psi _2(p-q)=\ba $, 
namely, $\degp \psi _2(p-q)=\ba $ and 
$\degp \psi _i(p-q)\prec \ba $ for all $i\ne 2$. 
In fact, we have 
$\ba =\degp \psi _2(p)\succ \degp \psi _i(p)$ for all $i\ne 2$ 
since $p\in I_{\p ,2}$, 
and 
$\ba \succ \mu (\ba )=\degp \psi _1(q)\succeq \degp \psi _j(q)$ 
for all $j$ since $q\in I_{\p ,1}$. 
This contradicts the minimality of $\mu (\ba )$.

(ii) Suppose that $\mu (\ba )=\mu (\ba ')$ for some 
$\ba ,\ba '\in \degp \psi _2(I_{\p ,2})$ with $\ba \succ \ba '$. 
Choose $p,p'\in I_{\p ,2}$ 
so that $\degp \psi _2(p)=\ba $, $\degp \psi _2(p')=\ba '$ 
and $\rinp \psi _1(p)=\rinp \psi _1(p')=\x ^{\mu (\ba )}$. 
Then, 
we have $\degp \psi _1(p-p')\prec \mu (\ba )$. 
We can also check that $p-p'\in I_{\p ,2}$ and $\degp \psi _2(p-p')=\ba $ 
as in the proof of (i), 
by noting 
$\ba =\degp \psi _2(p)\succ \degp \psi _i(p)$ 
for all $i\ne 2$, 
and $\ba \succ \ba '=\degp \psi _2(p')\succeq \degp \psi _j(p')$ 
for all $j$. 
This contradicts the minimality of $\mu (\ba )$.

(iii) follows from (ii) and $(2^{\circ })$. 
\end{proof}

Now, 
we can give a practical sufficient condition 
for $\p $-incompleteness.

\begin{thm}\label{thm:incompleteness criterion}
Let $\p \in \Omega $. 
If there exists a finite subset $\mathscr{F}$ of $\Z ^n$ 
such that the following conditions 
hold for each $\p '\in o(\cW _{\mathscr{F},\p })$, 
then $(\psi _i)_{i=1}^r$ is $\p $-incomplete.

\nd $(1)$ $\Rn (\deg _{\p '}\psi _1(R))=\Rn (\degp \psi _1(R))$. 

\nd $(2)$ There exists $\min _{\p '}M^{\ba }_{\p '}$ 
for each $\ba \in \deg _{\p '}\psi _2(I_{\p ',2})$. 

\end{thm}
\begin{proof}
Suppose that $(\psi _i)_{i=1}^r$ is $\p $-complete. 
Then, 
we have $\Rn (\degp \psi _1(R))=\sum _{i=1}^s\Rn (\degp \psi _1(p_i))$ 
for some $p_1,\ldots ,p_s\in I_{\p ,1}$, 
and $I_{\p ,2}\ne \emptyset $. 
Pick any $p_0\in I_{\p ,2}$, 
and set 
$\mathscr{G}
:=\mathscr{F}\cup \bigcup _{i=1}^r\bigcup _{j=0}^s\supp \psi _i(p_j)$.

Now, 
let $\p '\in o(\cW _{\mathscr{G},\p })\subset o(\cW _{\mathscr{F},\p })$. 
Then, we have 
$p_i\in I_{\p ',1}$ and $\degpp \psi _1(p_i)=\degp \psi _1(p_i)$ 
for $i=1,\ldots ,s$, 
and $p_0\in I_{\p ',2}$. 
Since $\p '\in o(\cW _{\mathscr{F},\p })$, 
we also have 
\begin{align*}
&\Rn (\degpp \psi _1(R))=\Rn (\degp \psi _1(R))\\
&\quad =\sum _{i=1}^s\Rn (\degp \psi _1(p_i)) 
=\sum _{i=1}^s\Rn (\degpp \psi _1(p_i))
\end{align*}
by (1). 
Hence, 
$\degpp \psi _1(R)\sm \degpp \psi (I)$ is a finite set 
because of Lemma~\ref{lem:FP} (ii). 
On the other hand, 
since $I_{\p ',2}\ne \emptyset $, 
we know by (2) and Lemma~\ref{lem:IP} (iii) that 
$\degpp \psi _1(R)\sm \degpp \psi (I)$ is an infinite set. 
This is a contradiction. 
\end{proof}

\begin{proof}[Proof of Proposition~$\ref{prop:key}$]
In view of (3$^\circ $), 
we may assume that $\degp \psi _1(R)$ is finitely generated. 
Since 
$\min _{\p }(\degp \psi _1(R))=\zero $ by assumption, 
we can find a finite subset $\mathscr{F}$ of $\Z ^n$ 
as in Proposition~\ref{prop:well-ordered} for $P:=\psi _1(R)$. 
We claim that (1) and (2) in Theorem~\ref{thm:incompleteness criterion} 
hold for each $\p '\in o(\cW _{\mathscr{F},\p })$. 
Actually, 
(1) follows from 
$\degpp P=\degp P$, 
and (2) follows from its well-orderedness. 
Therefore, 
$(\psi _i)_{i=1}^r$ is $\p $-incomplete 
thanks to Theorem~\ref{thm:incompleteness criterion}. 
\end{proof}

\section{Examples}\label{sect:example}
\setcounter{equation}{0}

In this section, 
we give some examples of interest. 
We also discuss the cardinality of 
the set of initial algebras 
shortly. 
For all $\mathscr{A}$ given in this section, 
$\rinp \mathscr{A}$ are not finitely generated 
if $\p $ is a monomial order 
thanks to Corollary~\ref{cor:main}.

First, 
we give an example in which 
$\degp \mathscr{A}$ is finitely generated 
for some $\p \in \Omega $.

\begin{example}\rm
Assume that $n=2$, 
and let $C$, $C_1$ and $C_2$ be as in Example~\ref{ex:basic}. 
Set $I:=(x^2+x)k[x]$ and $R:=I+k$. 
Then, 
(A1) is satisfied, 
because $k$-subalgebras of $k[x]$ are always finitely generated. 
We define homomorphisms $\phi _1',\phi _2':k[x]\to \kxx $ 
of $k$-algebras by $\phi _1'(x)=-x_1-1$ and $\phi _2'(x)=x_2$. 
Then, $\phi _i:=\phi _i'|_R$ for $i=1,2$ satisfy (A2), 
since $\phi _i(x^2+x)=x_i^2+x_i$. 
Take any $U\subset \kxx $ as in (A3).

Now, 
let $\p \in \Omega $ be such that 
$\e _1,\e _2\prec \zero $ 
and $l_1\e _2\prec \e _1$ 
and $l_2\e _1\prec \e _2$ 
for some $l_1,l_2\ge 1$, 
say $\p =o(-\w )$ 
with $\w \in (\Rp )^2\cap \cW $. 
We show that $\e _1,\e _2\in \deg _{\p }\mathscr{A}$, 
which implies $\deg _{\p }\mathscr{A}=\Zn \e _1+\Zn \e _2$. 
We have $\degp f=\e _1$, 
where 
$$
f:=\phi ((x^2+x)x^{l_1})
=(x_1^2+x_1)(-x_1-1)^{l_1}+(x_2^2+x_2)x_2^{l_1}
\in \mathscr{A}. 
$$
Take $c_1,\ldots ,c_{l_2-1}\in k$ 
so that $x_1^2+x_1-\sum _{i=1}^{l_2-1}c_i(x_1^2+x_1)^i(x_1+1)
\in x_1^{l_2}k[x_1]$ 
and set 
$p:=x^2+x+\sum _{i=1}^{l_2-1}c_i(x^2+x)^ix\in I$. 
Then, we have $\degp \phi (p)=\degp \phi _2(p)=\e _2$. 
This proves $\e _1,\e _2\in \deg _{\p }\mathscr{A}$. 
\end{example}

Examples~\ref{example:Hanoi} and \ref{example:autom} below 
are generalizations of Example~\ref{ex:RS}.

\begin{example}\label{example:Hanoi}\rm 
Let $C\subset (\Rn )^n$ be a rational polyhedral cone 
which is not contained in a line, 
and let $\ba _1,\ldots ,\ba _r\in (\Zn )^n\cap C$ 
be minimal generators of $C$, 
where $r\ge 2$. 
Take $C_i:=\Rn \e _i$ for $i=1,\ldots ,r$. 
Then, we have $C_i^{\circ }:=\Rp \e _i$ for each $i$. 
Let $R=k[x]$ and $I=xk[x]$, 
and define $\phi _i:R\ni p(x)\mapsto p(\x ^{\ba _i})\in \kx $ 
for $i=1,\ldots ,r$. 
Then, (A1) and (A2) hold. 
Take $U\subset \kx $ to be a finite set of monomials as in (A3). 
In this case, we have 
$$
\mathscr{A}=k+\phi (I)+J
=k+\sum _{l\ge 1}k(\x ^{l\ba _1}+\cdots +\x ^{l\ba _r})+J, 
$$
where $J$ is the ideal of 
$\mathscr{B}=k[\{ \x ^{\ba _1},\ldots ,\x ^{\ba _r}\} \cup U]$ 
generated by $\x ^{\ba _i+\ba _j}$ for $1\le i<j\le r$ and $U$. 
Now, pick any $\p \in \Omega $. 
Since $J$ is generated by monomials as a $k$-vector space, 
we have $\rinp J=J$. 
Hence, we know that 
$\rinp \mathscr{A}=k+\sum _{l\ge 1}k\x ^{l\ba _i}+J
=k[\x ^{\ba _i}]+J$ by (\ref{eq:deg A}), 
where $1\le i\le r$ is such that 
$\ba _i=\max _{\p }\{ \ba _1,\ldots ,\ba _r\} $. 
Conversely, 
for each $i$ and fixed $\w \in \cW $, 
we can find $\lambda \gg 0$ such that 
$\w -\lambda \omega _i\in \cW $ 
and 
$\ba _i=\max _{\p _i}\{ \ba _1,\ldots ,\ba _r\} $, 
where $\omega _i$ is a normal vector of the face $C_i$ of $C$ 
and $\p _i:=o(\w -\lambda \omega _i)$, 
since $\{ \lambda \in \Rp \mid \w -\lambda \omega _i\not\in \cW \} $ 
is at most countable. 
Therefore, we get $\# \{ \rinp \mathscr{A}\mid \p \in \Omega \} =r$. 
\end{example}

Finally, 
we give an example 
for which the set of initial algebras is continuum. 
We note that $\# \Omega =\# \R $ if $n\ge 2$, 
since $\# o(\cW )=\# \R $, 
and each element of $\Omega $ is determined by 
at most $n$ elements of $\R ^n$ (cf.~\cite{Robbiano}). 
Hence, 
$\# \{ \rinp A\mid \p \in \Omega \} \le \# \R $ 
holds for any $A\subset \kxx $.

The example is realized 
in $\kxy :=k[x_1,\ldots ,x_n,y_1,\ldots ,y_n]$, 
the polynomial ring in $2n$ variables over $k$. 
We fix a monomial order $\preceq $ on $\kxy $. 
For each $\lambda >0$, 
we set $\w _{\lambda }:=(\e _1,\lambda \e _n)\in \R ^n\times \R ^n$. 
Then, 
we can define a monomial order $\preceq _{\lambda }$ on $\kxy $ 
by $\ba \preceq _{\lambda }\bb $ if 
$\w _{\lambda }\cdot \ba <\w _{\lambda }\cdot \bb $, 
or $\w _{\lambda }\cdot \ba =\w _{\lambda }\cdot \bb $ 
and $\ba \preceq \bb $ for each $\ba ,\bb \in \Z ^n\times \Z ^n$.

\begin{example}\label{example:autom}\rm 
Let $C:=(\Rn )^n\times (\Rn )^n$, 
$C_1:=(\Rn )^n\times \bzs $ 
and $C_2:=\bzs \times (\Rn )^n$. 
Then, 
we have $C_i^{\circ }=C_i\sm \{ (\zero ,\zero )\} $ for $i=1,2$. 
Let $R=\kx $ and $I=(x_1,\ldots ,x_n)$, 
which satisfy (A1). 
We pick an automorphism $\theta $ 
of the $k$-algebra $\kx $ with $\theta (I)=I$, 
and define homomorphisms 
$\phi _1,\phi _2:R\to \kxy $ 
of $k$-algebras 
by $\phi _1(x_i)=\theta (x_i)$ 
and $\phi _2(x_i)=y_i$ for $i=1,\ldots ,n$. 
Since $\phi _j(x_i)$'s have no constant terms, 
(A2) is satisfied. 
Set $U:=\emptyset $. 
Then, 
noting $\phi _1(I)=\theta (I)=I$, 
we can check that $\mathscr{B}=\kxy $, 
$J=(\{ x_iy_j\mid i,j\in \{ 1,\ldots ,n\}\} )$, 
and 
$$
\mathscr{A}=k[\theta (x_1)+y_1,\ldots ,\theta (x_n)+y_n]+J. 
$$

Now, assume that $n\ge 2$ and $\theta =\id _{\kx }$. 
We show that 
$\rin _{\p _{\lambda }}\!\mathscr{A}
\ne \rin _{\p _{\mu }}\!\mathscr{A}$ 
for each $\lambda >\mu >0$. 
This implies that 
$$
\# \{ \rin _{\p _{\lambda }}\!\mathscr{A}\mid \lambda >0\} 
=\# \{ \rinp \mathscr{A}\mid \p \in \Omega _0\} 
=\# \{ \rinp \mathscr{A}\mid \p \in \Omega \} 
=\#\R .
$$
Suppose that 
$\rin _{\p _{\lambda }}\!\mathscr{A}
=\rin _{\p _{\mu }}\!\mathscr{A}$ 
for some $\lambda >\mu >0$. 
Then, 
we have 
$\rin _{\p _{\lambda }}\!\phi (I)
=\rin _{\p _{\mu }}\!\phi (I)$ 
by (\ref{eq:deg A}). 
The $k$-vector space 
$\phi (I)=\{ p+\phi _2(p)\mid p\in I\} $ 
has a basis 
$$
\{ f_{\bi }
:=x_1^{i_1}\cdots x_n^{i_n}+y_1^{i_1}\cdots y_n^{i_n}\mid 
\bi =(i_1,\ldots ,i_n)\in (\Zn )^n\sm \bzs \} . 
$$
Since $\supp f_{\bi }\cap \supp f_{\bj }=\emptyset $ if $\bi \ne \bj $, 
it follows that 
$\rin _{\p _{\lambda }}f_{\bi }
=\rin _{\p _{\mu }}f_{\bi }$ 
for all $\bi \in (\Zn )^n\sm \bzs $. 
However, 
if $\bi =a\e _1+b\e _n$ for 
$a,b\in \Zp $ with $\lambda >a/b>\mu $, 
then we have $\rin _{\p _\lambda }f_{\bi }=y_1^ay_n^b$ 
and $\rin _{\p _\mu }f_{\bi }=x_1^ax_n^b$. 
This is a contradiction. 
\end{example}

The author proved that, 
if the transcendence degree of $R$ over $k$ is at least two, 
and if there exist $1\le i<j\le r$ 
such that the set $\Omega '$ of 
$\p \in \Omega $ satisfying 
$I_{\p ,i}\ne \emptyset $, 
$I_{\p ,j}\ne \emptyset $ 
and $\min _{\p }(\supp \phi _i(R))=\zero $ 
is not empty, 
then 
$\# \{ \rinp \mathscr{A}\mid \p \in \Omega '\} =\# \R$. 
The proof of this result is rather long, 
and it will appear in our next paper.


\begin{thebibliography}{00}

\bibitem{Anderson}
S. Anderson, A. Smith, P. Stewart, M. Tesemma and J. Usatine, 
A topological structure on certain initial algebras, Topology Appl. {\bf 180} (2015), 199--208.

\bibitem{AM}
M. F. Atiyah\ and\ I. G. Macdonald, 
{\it Introduction to commutative algebra}, 
Addison-Wesley Publishing Co., Reading, MA, 1969. 


\bibitem{Reich09}
A. Duncan\ and\ Z. Reichstein, 
Sagbi bases for rings of invariant Laurent polynomials, 
Proc. Amer. Math. Soc. {\bf 137} (2009), no.~3, 835--844. 



\bibitem{Gob}
M. G\"{o}bel, 
\textit{A constructive description of SAGBI bases for polynomial invariants of permutation groups}, J. Symbolic Comput. {\bf 26} (1998), no.~3, 261--272.


\bibitem{KM}
D. Kapur\ and\ K. Madlener, 
A completion procedure for computing a canonical basis for a $k$-subalgebra, 
in {\it Computers and mathematics (Cambridge, MA, 1989)}, 
1--11, Springer, New York.


\bibitem{KurodaRIMS00}
S. Kuroda, 
The infiniteness of the SAGBI 
(subalgebra analogue to Gr\"{o}bner bases for ideals) 
bases for certain invariant rings, 
S\={u}rikaisekikenky\={u}sho K\={o}ky\={u}roku 
No. 1175 (2000), 51--62.



\bibitem{Kuroda02}
S. Kuroda, 
\textit{The infiniteness of the SAGBI bases for certain invariant rings}, 
Osaka J. Math. {\bf 39} (2002), no.~3, 665--680. 


\bibitem{JvdK}
S. Kuroda, Initial algebras and the Jung--van der Kulk theorem, 
in {\it Affine algebraic geometry}, 
193--204, CRM Proc. Lecture Notes, 54, Amer. Math. Soc., Providence, RI.

\bibitem{Sugaku}
S. Kuroda, 
\textit{Recent developments in polynomial automorphisms: 
the solution of Nagata's conjecture and afterwards}, 
Sugaku Expositions {\bf 29} (2016), no.~2, 177--201.


\bibitem{Oda}
T. Oda, 
{\it Convex bodies and algebraic geometry}, 
translated from the Japanese, 
Ergebnisse der Mathematik und ihrer Grenzgebiete (3), 15, 
Springer-Verlag, Berlin, 1988.


\bibitem{Reich03}
Z. Reichstein, 
SAGBI bases in rings of multiplicative invariants, 
Comment. Math. Helv. {\bf 78} (2003), no.~1, 185--202.


\bibitem{Robbiano}
L. Robbiano, 
Term orderings on the polynomial ring, 
in {\it EUROCAL '85, Vol. 2 (Linz, 1985)}, 513--517, 
Lecture Notes in Comput. Sci., 204, Springer, Berlin. 


\bibitem{RS}
L.~Robbiano and M.~Sweedler, 
Subalgebra bases, 
in Commutative Algebra 
(W. Bruns and A. Simis, eds.) 61--87, 
Lecture Notes in Math.\ {\bf 1430}, 
Springer, Berlin, Heidelberg, New York, Tokyo, 1988. 


\bibitem{LP}
A. Schrijver, 
{\it Theory of linear and integer programming}, 
Wiley-Interscience Series in Discrete Mathematics, 
Wiley, Chichester, 1986. 


\bibitem{Schwartz}
N. Schwartz, 
Stability of Gr\"{o}bner bases, 
J. Pure Appl. Algebra {\bf 53} (1988), no.~1-2, 171--186.

\bibitem{Sturmfels}
B. Sturmfels, {\it Gr\"{o}bner bases and convex polytopes}, 
University Lecture Series, 8, 
American Mathematical Society, Providence, RI, 1996.


\bibitem{Tesemma}
M. Tesemma, 
On initial algebras of multiplicative invariants, 
J. Algebra {\bf 320} (2008), no.~11, 3851--3865. 


\bibitem{TT}
N. M. Thi\'{e}ry\ and\ S. Thomass\'{e}, 
Convex cones and SAGBI bases of permutation invariants, 
in Invariant theory in all characteristics, 259--263, CRM Proc. Lecture Notes, 35, Amer. Math. Soc., Providence, RI. 

\end{thebibliography}
\end{document}